\documentclass[11pt,leqno,twoside]{amsart}
\usepackage{amssymb,amsmath,amsthm,soul,color}
\usepackage{t1enc}
\usepackage[cp1250]{inputenc}
\usepackage{a4,indentfirst,latexsym,graphics,graphicx}
\usepackage{graphics}
\usepackage{mathrsfs}
\usepackage{cite,enumitem,graphicx}
\usepackage[colorlinks=true,urlcolor=blue,
citecolor=red,linkcolor=blue,linktocpage,pdfpagelabels,
bookmarksnumbered,bookmarksopen]{hyperref}
\usepackage[english]{babel}
\usepackage[left=2.61cm,right=2.61cm,top=2.72cm,bottom=2.72cm]{geometry}
\usepackage[metapost]{mfpic}
\usepackage[hyperpageref]{backref}
\usepackage[colorinlistoftodos]{todonotes}
\usepackage[normalem]{ulem}
\usepackage{pdfsync}

\makeatletter
\providecommand\@dotsep{5}
\def\listtodoname{List of Todos}
\def\listoftodos{\@starttoc{tdo}\listtodoname}
\makeatother

\numberwithin{equation}{section}
\newtheorem{theorem}{Theorem}[section]
\newtheorem{proposition}[theorem]{Proposition}
\newtheorem{lemma}[theorem]{Lemma}
\newtheorem{corollary}[theorem]{Corollary}
\newtheorem{definition}[theorem]{Definition}
\newtheorem{remark}[theorem]{Remark}
\newtheorem{Ex}[theorem]{Example}

\title[Compactness and nonradial solutions]{
 A group theoretic proof of a compactness lemma \\ and  existence of nonradial solutions for \\
semilinear elliptic equations}
\author[L. Biliotti]{Leonardo Biliotti}
\author[G. Siciliano]{Gaetano Siciliano}

\address[L. Biliotti]{\newline\indent
	Dipartimento di Scienze Matematiche, Fisiche e Informatiche,
	\newline\indent
         Universit\`a di Parma
	\newline\indent
	Parco Area delle Scienze 7/A,
        43124 Parma, Italy}
\email{\href{mailto:leonardo.biliotti@unipr.it}{leonardo.biliotti@unipr.it}}

\address[G. Siciliano]{\newline\indent
	Departamento de Matem\'atica - Instituto de Matem\'atica e Estat\'istica
	\newline\indent
	Universidade de S\~ao Paulo
	\newline\indent
	Rua do Mat\~ao 1010,  05508-090  S\~ao Paulo, Brazil}
\email{\href{mailto:sicilian@ime.usp.br}{sicilian@ime.usp.br}}

\thanks{The first author was partially supported by PRIN  2015
   ``Variet\`a reali e complesse: geometria, topologia e analisi armonica'' and GNSAGA INdAM.
   The second author is supported by Capes, CNPq n.304660/2018-3 and Fapesp n.2018/17264-4.}
\subjclass[2010]{
35J20,  
35J61,  
22E60. 	
}
\keywords{Compactness lemma, existence of nonradial solutions, symmetric spaces}

\begin{document}

\maketitle
\begin{abstract}
Symmetry plays a basic role in variational problems (settled e.g. in $\mathbb R^{n}$ or in a more general manifold), for example to deal with the lack
of compactness which naturally appear when the problem is invariant under the action of a noncompact group.
In $\mathbb R^n$,  a compactness result for invariant functions with respect to a subgroup $G$ of $\mathrm{O}(n)$ has been proved under the condition that the $G$ action on $\mathbb R^n$ is compatible,
see \cite{willem}.
As a first result we generalize this and show  here that the compactness is recovered for particular  subgroups of  the isometry group of a Riemannian manifold.
We investigate also  isometric action on Hadamard manifold $(M,g)$ proving that a large class of subgroups of $\mathrm{Iso}(M,g)$ is compatible. As an application we get a compactness result for
``invariant'' functions which allows us to prove the existence of nonradial solutions for a classical scalar equation and for a nonlocal fractional equation on $\mathbb R^n$ for $n=3$ and $n=5$,
improving some results known in the literature.
Finally, we prove the existence of nonradial invariant functions such that a compactness result holds for some symmetric spaces of non compact type.
\end{abstract}

\bigskip

\begin{center}
\begin{minipage}{12cm}
\tableofcontents
\end{minipage}
\end{center}

\medskip

\maketitle
%
\section{Introduction}
It is known that many interesting partial differential equations in $\mathbb R^{n}$ are invariant under
the orthogonal group $\textrm O(n)$
so that it makes sense to find solutions which respect this symmetry, i.e. are radial.
These solutions are physically interesting and indeed
in scalar field theory they are also called {\sl particle-like}.
Particularly interesting is the case when the equations are variational: i.e. a smooth functional (called {\sl the energy functional}) on
Banach or Hilbert space $X$ can be defined in such a way that its critical points give exactly the solutions
of the equations;  very often the restriction of this functional to the subspace of radial functions $X_{\textrm O(n)}$
is even ``natural'', in the sense of the {\sl Palais' Criticality Symmetric Principle} \cite{PA}:
one roughly speaking says that critical symmetric points are symmetric critical points.
The advantage of working in the subspace $X_{\textrm O(n)}$ is that its elements may have
additional properties which enable to recover a compactness condition (as required by many abstract theorems in Critical Point Theory) that the functional have to satisfy in order to guarantee the existence of critical symmetric points (see below for a specific problem). From a functional point of view, this compactness is a consequence of the
compact embedding of Sobolev spaces into Lebesgue spaces.

\medskip

A natural problem that arises is then the search of nonradial solutions for such equations,
and indeed the main difficulty is exactly to guarantee that the solutions found are effectively non radial.

This topic has attracted much attention
and the  existence of a nonradial solutions have been intensively investigated by many authors.

Particularly interesting for our purpose
 is the work by
Bartsch and Willem \cite{BW} where the authors consider the following
equation
\begin{equation}\label{eq:bW}
-\Delta u + b(|x|) u =f(|x|,u) \quad \text{ in } \mathbb R^{n}, \ n\geq3
\end{equation}
under suitable assumptions on $b$ and $f$ and look for
non radial solutions. The approach of the authors consists in
restricting the energy functional, let us say $\phi$, which is naturally defined in
the Sobolev space $H^{1}(\mathbb R^{n})$, to the subspace $H^{1}_{G}(\mathbb R^{n})$
of fixed points for a suitable group action $G$ which does not contains radial functions (except of course the null function).
Roughly speaking and without entering  in details here, we can say that  the group $G$ is generated by
\begin{itemize}
\item[(i)] functions which are ``radial in groups of variables'', that is, they are invariant for the action induced by
the subgroup of $\textrm O(n)$:
$$\textrm O(m)\times \textrm O(m) \times\textrm O(n-2m),$$
\item[(ii)]  and by functions which are invariant  by a ``suitable action'' induced by
$$\tau \cdot (x_{1}, x_{2}, x_{3}) = (x_{2}, x_{1}, x_{3}),$$ where $(x_{1}, x_{2}, x_{3}) \in \mathbb R^m \times \mathbb R^{m} \times
\mathbb R^{n-2m}.$
\end{itemize}
Note that $\tau^{2}=e$, the identity. The success of the method is based on the fact that,
as proved by  Lions  in \cite{L}, the set of functions in $H^{1}(\mathbb R^{n})$ satisfying (i) has compact embedding
into $L^{p}$ spaces and the
same holds for $H^{1}_{G}(\mathbb R^{n})$, that is, when also the action of $\tau$
is taken into account. As a consequence,
the energy functional restricted to $H^{1}_{G}(\mathbb R^{n})$ satisfies
the {\sl Palais-Smale condition}: any sequence $\{u_{k}\}\subset H^{1}_{G}(\mathbb R^{n})$
such that
$$\{\phi(u_{k})\} \ \text{is convergent and }  \phi'(u_{k})\to 0$$
has a convergent subsequence.
However, in order to consider (ii) and then guarantee that there are not nontrivial  radial function
in $H^{1}_{G}(\mathbb R^{n})$, the authors assume that $2\leq m\leq n/2$ and $2m\neq n-1$
which forces $n=4$ or $n\geq6$.

Another remarkable paper where nonradial solutions for an elliptic equation
are found by means of a similar strategy, is the one by d'Avenia \cite{P}.
Here  the author is interested in a so called {\sl Schr\"odinger-Poisson equation}
in $\mathbb R^{3}$ and  he restricts the energy functional to
the set of functions in $\mathbb R^{3}$ which are radial in the first two variables
and even in the third one, in order to guarantee that the solution found is not radial.
With our approach we find nonradial solutions in $\mathbb R^{3}$, (since they are
radial in the first two variables and periodic in the third one) and for $n\geq4$ the compact
embedding of our working space into $L^{p}$ permits to have nonradial solutions,
without any periodicity assumption.

%
%
%

\medskip

Looking at the group theoretic properties used in the previous papers
and motivated also by
\cite[Definition 1.22 p.16]{willem} where Willem gives  the definition
of {\sl compatible group} which permits to have the compact embedding
of ``invariant'' functions into $L^{p}$, we
try here to generalize and understand when the compact embedding
of Sobolev spaces of ``non radial'' functions into Lebesgue spaces holds.

Observe that in particular from \cite[Definition 1.22 p.16]{willem}
it follows that  compatible groups with $\mathbb R^{n}$ are
$G=\mathrm{O}(n)$ and $G=\mathrm{O}(N_1) \times \cdots \times \mathrm{O}(N_k)$, where $N_j \geq 2$, $j=1,\ldots,k$ and $\sum_{j=1}^k N_j=n$, and the compact embedding results of
Lions \cite{L} are recovered.

%
%
%


\medskip

\subsection{Main results: general statements}
Motivate by the cited papers,
we generalize here
the construction of the group action $G$ given in \cite{BW}
and investigate  isometric actions on $(\mathbb R^n,\langle \cdot , \cdot \rangle)$ and on its open $G$-invariant unbounded subsets
(being interested into compact embeddings, we will consider just unbounded subsets).
What we prove is that   a large class of subgroups of $\mathrm{Iso}(\mathbb R^n, \langle \cdot,\cdot \rangle )$  is compatible with $\mathbb R^{n}$ according to  Definition \ref{def-compatible}.
See Propositions \ref{toritto-discreto-generale}, \ref{toritto-errenne-discreto-generale} and \ref{toritto-sostanzioso}.
Here $\langle \cdot,\cdot\rangle$ denotes the canonical scalar product. Actually the more general case
of Riemannian manifold is treated.

As an application we prove the existence of nonradial solutions for the problem in \eqref{eq:bW}
for $n=3$ and $n=5$, thus extending the result of Bartsch and Willem in \cite{BW}.
To show the generality of the method and the range of applicability of our abstract results,
 we  show two more applications to systems of elliptic equations.
 The first one if  to the existence of nonradial solutions for a
 nonlocal fractional equation on $\mathbb R^n$ extending again to the case $n=3$ and $n=5$
a result in \cite{DSS} that was stated just for the cases $n=4$ and $n\geq6$.
The second one is the  existence  result of nonradial solutions in \cite{P}
 to the case of $\mathbb R^{n}, n\geq3$ (actually we show a multiplicity result).
However, since the statement of these theorems would imply many preliminary details and assumptions,
we prefer do not state them here but refer the reader directly to Theorems \ref{th:BW}, \ref{th:DSS} and \ref{th:P}
in Subsection \ref{subsec:appl}.

 \medskip

 After the euclidean case,  we investigate subgroups of the isometry group of an Hadamard manifold $(M,g)$. This means $(M,g)$ is a simply connected complete Riemannian manifold of nonpositive sectional curvature and so for any $p\in M$,  $\exp_p :T_p M \longrightarrow M$ is a diffeomorphism (see e.g. \cite{eberlein,helgason}). By Cartan Theorem (see e.g. \cite{helgason}) a compact group $G$ acting  isometrically on $(M,g)$  has a fixed point: call it $p$.  
 Then $\exp_p :T_p M \longrightarrow M$ becomes $G$-equivariant, where the $G$ action on $T_p M$ is the isotropy representation. Since by Rauch Theorem (see  \cite{docarmo}) the exponential map increases the distance, we prove that a large class of subgroups of $\mathrm{Iso}(M,g)$ are compatible (Proposition \ref{toritto-varieta}.)

 As an application  of this, we get a compactness result for invariants functions on an Hadamard manifold. We also point out that if a discrete and closed group $G$ acting isometrically on a Riemannian manifold $(M,g)$ is such that $M/G$ is compact then it is  compatible.
Finally we prove that the basic tool used in \cite{BW} in order to prove the existence of
a group action such that its fixed points have nonradial functions,  holds for some symmetric spaces of noncompact type. Roughly speaking, if $M=\mathrm{SL}(2n,\mathbb R)/\mathrm{SO}(2n)$, or $M$ is a the dual of the complex, real and quaternion projective space (see \cite{helgason}, then) we determine a compact subgroup $\widetilde H$ of the isometry group containing a closed subgroup of dimension bigger than one and of index two. This allow us to define an isometric action of $\widetilde H$ on $H^1 (M)$ such that the invariant functions are not radial unless $u=0$ and the embedding of the fixed points set with respect to $\widetilde H$ into $L^p (M)$ with $2<p<2^*$ is compact.

\medskip

The paper is organized as follows.

In Section \ref{sec:IA} we recall some facts on isometric actions on Riemannian manifolds
and give the definition of compatible group. The main result here is the general
Proposition \ref{toritto-discreto-generale}.

In Section \ref{sec:euclidean} we consider the special case of $\mathbb R^{n}$ where we give applications of our method to the existence of nonradial solutions for partial
differential equations. A fundamental tool in order to show the main results, Theorems \ref{th:BW}, \ref{th:DSS} and \ref{th:P}, is the technical Lemma \ref{toritto-tecnico}.

In Section \ref{sec:Hadamard} we consider the case of Hadamard manifold and in Section
\ref{sec:SSnoncompact} the case of symmetric spaces of noncompact type.

\medskip

\subsection*{Acknowledgments} The authors wish to thank Fabio Podest\`a and Jaroslaw Medersky
for interesting discussions and to point out Remark \ref{toritto-esempio}.

\section{Isometric actions on Riemannian manifolds}\label{sec:IA}
Let $(M,g)$ be a connected Riemannian manifold. We may introduce the distance on $M$ via the notion of length of curves, that we denote by $d$, and the topology of $(M,d)$ coincides with the manifold topology, see e.g. \cite{docarmo}. Let $\mathrm{Iso}(M,g)$ the group of isometries of $(M,g)$. It is well-known that any closed subgroup $G\subset \mathrm{Iso}(M,g)$ is a Lie group with respect to the compact open topology. In particular $\mathrm{Iso}(M,g)$ is a Lie group and if $M$ is compact then $\mathrm{Iso}(M,g)$ is compact as well, see \cite{kobayashi-transformation}. Moreover, the map
\[
G \times M \longrightarrow M, \qquad (f,p) \mapsto f(p),
\]
is differentiable and so defines a differential action on $M$ which is a proper action. This means that  the map
\[
G \times M \longrightarrow M \times M, \qquad (f,p) \mapsto (p,f(p)),
\]
is proper. By a well-known results it follows that for any $p\in M$, the isotropy subgroup
$$G_p=\{g\in G:\, gp=p\}\subset G$$ is compact, the orbit throughout $p$, i.e.,
$$G(p)=\{gp:\, g\in G\},$$ is a closed embedded submanifold of $M$ and the Slice Theorem and The Principal Orbit theorem hold.
For these facts we refer the reader e.g. to \cite{renatone-marcox}.

Denote  $\mathcal B_r(p)=\{q\in M: d(p,q)<r \}$, respectively $\mathcal S_r (p)=\{q\in M:\, d(p,q)=r\}$, the open ball of center $p$ and radius $r$, respectively the sphere of center $p$ and radius $r$.

A major role in what follows will be played by the next number.
\begin{definition}\label{def-compatible}
Let $G\subset \mathrm{Iso}(M,g)$ be a closed subgroup. For  $y\in M$ and  $r>0$, we define
\[
\mathrm{m}(y,r,G)=\mathrm{sup}\{n\in \mathbb N:\, \exists g_1,\ldots,g_n \in G: j\neq k \Rightarrow \mathcal B_r (g_j y) \cap \mathcal B_r (g_k y ) =\emptyset\}.
\]
Let $\Omega \subset M$ be an open $G$-invariant unbounded subset. We say that $\Omega$ is compatible with $G$
(or that $G$ is compatible with $\Omega$)
if there exists $r>0$  and $p\in M$ such that
\[
\lim_{\begin{array}{l} d(p,y) \to +\infty \\ d(\Omega,y)\leq r \end{array}}  m(y,r,G)=+\infty.
\]
\end{definition}
Let $q\in M$. Since $d(q,y) \mapsto +\infty$ if and only if $d(p,y) \to +\infty$, the above definition does not depend on $p\in M$.
Moreover, if $G\subset K$, then $\mathrm{m}(y,r,G)\leq \mathrm{m} (y,r,K)$. Hence if $\Omega$ is compatible with $G$ then it is compatible with $K$.
\begin{proposition}\label{toritto-discreto-generale}
Let $(M,g)$ be a connected noncompact Riemannian manifold and let $G\subset\mathrm{Iso}(M,g)$ be a closed and discrete subgroup with infinite elements.
Then there exists $r>0$ and a $G$-invariant open unbounded domain $\Omega$ of $(M,g)$ such that $\mathrm{m}(z,r,G)=+\infty$ for any $z\in \Omega$.
Moreover, if $M/G$ is compact, then there exists $r>0$ such that $\mathrm{m}(z,r,G)=+\infty$ for any $z\in M$ and so any open $G$-invariant unbounded domain of $M$ is compatible with $G$.
\end{proposition}
\begin{proof}
Let $p\in M$ be such that $G(p)$ is a principal orbit. We claim that $G_p=\{p\}$. Indeed, since the orbit throughout $p\in M$ is a discrete set, it follows that the slice representation coincides with the isotropy representation, i.e.,
\[
\rho:G_p \longrightarrow \mathrm{O}(T_p M), \qquad g \mapsto \mathrm d g_p.
\]
Since $G(p)$ is a principal orbit we have that $\rho$ is trivial and so if $g\in G_p$ then $g(p)=p$ and $\mathrm{d g}_p =\mathrm{Id}$. This implies that $g=\mathrm{Id}_M$, see \cite{docarmo}.

Applying the Slice Theorem there exists a $G$-invariant neighborhood $U$ of $p$ such that for any $q\in U$, $G_q=\{e\}$ and
\[
G(q) \cap U=G_p(q)=q.
\]
We may assume that $U=\mathcal B_r (p)=\{q\in M: d(p,q) <r\}$  for some $r>0$. We claim that for any $r'<r$ and for any $q\in \mathcal B_r (p)$, we have
\[
d(gq,hq) > r',
\]
whenever  $g\neq h$. Indeed, otherwise there exist $g,h\in G$ such that
\[
r> r'\geq d(gp,hp)=d(p,g^{-1}hp).
\]
Hence $g^{-1}h=\theta \in G_p=\{Id\}$ and so $h=g$. A contradiction.

Now, let any $p\in M$. Since $G_p$ is compact it follows that the cardinality of $G_p$ is finite. Applying the Slice Theorem, there exists $r>0$ such that for any $q\in \mathcal B_r (p)$ we have
\[
G(q) \cap \mathcal B_r (p)=G_p (q),
\]
and $G_q\subset G_p$. Therefore, if $gG_p \neq hG_p$, we have
\[
d(gq,hq)>r',
\]
for any $q\in \mathcal B_r (p)$ and for any $r'<r$. Indeed, as before, assume that there exists $g,h\in G$, such that $gG_p \neq hG_p$ and  $d(gq,hq)\leq r'$. Therefore $d(h^{-1}gq,q)\leq r' <r$ and so $h^{-1}g\in G_p$.  A contradiction.

Now, let $p\in M$. Given  $\epsilon >0$, we denote by $B_\epsilon =\{z\in T_p M:\, \parallel z \parallel <\epsilon\}$ and by $S_\epsilon =\{z\in T_p M:\, \parallel z \parallel =\epsilon\}$, where $\parallel \cdot \parallel=\sqrt{g(p)(\cdot,\cdot)}$, the ball of radius $r$, respectively the sphere of radius $r$, in $T_p M$.
Let $\epsilon >0$ such that $\exp_p:B_\epsilon \longrightarrow \mathcal B_\epsilon (p)$ is a diffeomorfism onto, see \cite{docarmo}. Let $\alpha<\epsilon$ and let $q\in S_\alpha (p)=\exp(S_\alpha)$. We have proved that there exists $r(q)>0$ such that for any $z\in \mathcal B_{r(q)} (q)$, we have
\[
d(gz,hz) >r(q),
\]
whenever  $gG_q \neq h G_q$. Since $\mathcal S_\alpha (p)$ is compact, there exists $q_1\,\ldots, q_m \in \mathcal S_\alpha (p)$ and $r_1(q_1), \ldots r_m (q) >0$ such that
\[
\mathcal S_\alpha (r) \subset \mathcal B_{r_1(q)} (q_1 ) \bigcup \cdots \bigcup \mathcal B_{r_m(q)} (q_m )
\]
and for any $z\in \mathcal B_{r_j(q)} (q_j )$ we have
\[
d(g z,h z) >r(q_j).
\]
whenever $g G_{q_j} \neq h G_{q_j}$, for $j=1,\ldots,m$. Applying the triangle inequality, we have $\mathcal B_{\frac{r}{2}} (gz) \cap \mathcal B_{\frac{r}{2}} (hz)=\emptyset $ for any $g,h\in G$ such that $gG_{g_j}\neq hG_{g_j}$ and so for infinite $g,h\in G$. Therefore $\mathrm{m}(z,r/2,G)=+\infty$. This holds for any $z\in \mathcal B_{r_1(q)} (q_1 ) \bigcup \cdots \bigcup \mathcal B_{r_m(q)} (q_m )$.

Let $\Omega=G \left(\mathcal B_{r_1(q)} (q_1 ) \bigcup \cdots \bigcup \mathcal B_{r_m(q)} (q_m ) \right)$. Then
$\Omega$ is a $G$-invariant  unbounded domain of $(M,g)$, i.e., $d:\Omega \times \Omega \longrightarrow \mathbb R$ is not bounded above and, finally, for any $z\in \Omega$, we have
$\mathrm{m}(z,r/2,G)=+\infty$.

For the second part,
assume that $M/G$ is compact. Let $p\in M$. We have proved that there exists $r(p)>0$ such that for every $q\in \mathcal B_r (p)$, we have
\[
G(q) \cap \mathcal B_{r(p)} (p)=G_p (q),
\]
and $d(gq,hq) > r$ whenever $gG_p \neq hG_q$. Therefore $\Omega_p =G \mathcal B_{r(p)} (p)$ is an open $G$-invariant unbounded domain of $M$ such that for any $z\in \Omega$, we have
\[
d(gz,hz) > r,
\]
for infinite $g,h\in G$. Since
\[
M/G=\bigcup_{p\in M} \Omega_p /G,
\]
and $M/G$ is compact, there exits $p_1,\ldots,p_m \in M$ such that
\[
M=\Omega_{p_1} \cup \cdots \cup \Omega_{p_m}.
\]
Pick $r=\frac{1}{2}\mathrm{min}\{r (p_1),\ldots,r(p_m) \}$. Then for every $z\in M$, we have
\[
\mathrm{m}(z,r,G)=+\infty
\]
which concludes the proof.
\end{proof}
\begin{remark}
If $(M,g)$ is compact then $\mathrm{Iso}(M,g)$ is compact and so any closed and discrete subgroup $G\subset\mathrm{Iso}( M,g)$ is finite. This means that the condition $(M,g)$ is noncompact it is not a technical condition.
\end{remark}

\section{The case of $\mathbb R^{n}$ and application to PDEs}\label{sec:euclidean}
Now, we investigate isometric action on $\mathbb R^n$ endowed by the canonical scalar product $\langle \cdot,\cdot \rangle$. In the sequel we continue to denote by $d(x,y)=\parallel x-y\parallel=\sqrt{\langle x-y,x-y\rangle}$, i.e.,  the distance induced by $\langle \cdot,\cdot \rangle$ and by $B_r (p)=\{z\in \mathbb R^n:\, \parallel z-p \parallel <r\}$ the euclidian open ball.

Given $\Omega\subset \mathbb R^{n}$ let $H^{1}(\Omega)$ be the usual Sobolev space
and  $H^{1}_{0}(\Omega)$ be the closure of $C_{0}^{\infty}(\Omega)$ in $H^{1}(\Omega)$.

\begin{proposition}\label{toritto-errenne-discreto-generale}
Let $G\subset \mathbb R^n$ be a discrete non finite subgroup of $\mathrm{Iso}(\mathbb R^n,\langle \cdot,\cdot \rangle)=\mathrm{O}(n) \ltimes \mathbb R^n$ contained in $\mathbb R^n$. Then there exists $r>0$  such that $\mathrm{m}(z,r,G)=+\infty$ for any $z\in \mathbb R^n$.
\end{proposition}
\begin{proof}
Since $G\subset \mathbb R^n$ is discrete, there exists $\alpha>0$ such that $G \cap B_\alpha (0)=\{0\}$.  Let $g,h\in G$. Then $gz=z+t(g)$ and $hz=z+t(h)$, respectively, and so
\[
d(gz,hz)=\parallel t(g)-t(h) \parallel \geq \alpha.
\]
Pick $r=\frac{1}{4}\alpha$. Then for every $z\in \mathbb R^n$, and for every $g,h\in G$, we have
\[
B_r (gz) \cap B_r (hz)=\emptyset,
\]
whenever $g\neq h$ concluding the proof.
\end{proof}
Now, assume that $G$ is a connected and  closed subgroup of $\mathrm{Iso}(\mathbb R^n, \langle \cdot, \cdot \rangle)$ of dimension bigger than one. Let $(\mathbb{R}^n)^G=\{p\in \mathbb R^n: G(p)=p\}$ be the fixed points set. It is well-known that $(\mathbb{R}^n)^G$ 
is totally geodesic, see \cite{renatone-marcox}. Hence if it was not empty then it would be a closed subspace. In the sequel we assume that $(\mathbb{R}^n)^G$ is empty or is reduced to $\{0\}$. In particular for any nonzero vector $y\in \mathbb R^n$, the orbit $G(y)$ has dimension at least one

Let $y \in \mathbb{R}^n\setminus\{0\}$ and let $\lambda \in \mathbb R$. Let $g\in G$ be such that $g\notin G_y$. Then $g(q)=Aq+v$ for some $A\in \mathrm{O}(n)$ and some $v\in \mathbb R^n$. Assume firstly that $Ay\neq y$. Then
\[
\lim_{\lambda \mapsto +\infty} d(\lambda y, \lambda Ay+v)^2=\lambda^2 \parallel Ay-y \parallel^2 -2 \lambda \langle Ay-y,v \rangle +  \parallel v \parallel^2 =\infty.
\]
Let $r>0$ and let $n\in \mathbb N$. Let $\lambda_o >0$ be such that for any $\lambda >\lambda_o$ we have $d(\lambda y, A \lambda y+v)^2 > (2nr)^2$.
Since $G(\lambda y)$ is a closed submanifold, it follows $(G(\lambda y),\langle \cdot, \cdot \rangle)$ is a complete Riemannian manifold. By a Theorem of Hopf-Rinow \cite{docarmo} there exists at least one minimizing geodesic $\gamma:[0,l] \longrightarrow G(\lambda y)$ parameterized with the arc length and satisfying $\gamma(0)=\lambda y$, $\gamma(l)=g(\lambda y)$ and  $L(\gamma)=l=d^{G(\lambda y)} (g (\lambda y),\lambda y) \geq d(\lambda y, g(\lambda y))> 2nr$. Let
\[
f:[0,l] \longrightarrow [0,d(\lambda y,g(\lambda y))], \qquad f(t)=d(\lambda y, \gamma(t)).
\]
Then $f$ is continuous and surjective since the image is connected and it contains $0$ and $d(\lambda y, g(\lambda y))$. Pick
\[
t_1=\mathrm{sup}\{t\in [0,l]:\, d(\lambda y, \gamma(t))=2r\}.
\]
Then $0<t_1<d(\lambda y, g(\lambda y))$ and for any $t\geq t_1$, we have $f(t)=d(\lambda y,\gamma(t)) > 2r$. Indeed, assume that there exists $t'>t_1$ such that $f(t')<2r$. Since $f(l)=d(\lambda y , g(\lambda y))>2r$, it follows that there exists $t''>t'>t_1$ such that $f(t'')=2r$ which is a contradiction. Moreover,  since
\[
d(\lambda y, g (\lambda y)) \leq d(\lambda y, \gamma (t_1 ) ) + d(\gamma(t_1), g(\lambda y) ),
\]
it follows
\[
d(\gamma(t_1), g(\lambda y)) \geq d(\lambda y , g (\lambda y ) )- 2r \geq 2(n-1)r.
\]
Therefore, we are able to iterate this procedure at least $n-1$ times, proving that there exist $0=t_0<t_1 < \ldots <t_{n-1}$ such that
\[
d(\gamma (t_i), \gamma(t_j)) \geq 2r,
\]
whenever $i \neq j$. This implies, again applying the triangle inequality, $B_{r}(\gamma(t_i)) \cap B_{r}(\gamma(t_j))=\emptyset$ for any $i\neq j$. Since $\gamma$ is a minimizing geodesic,  denoting  $\gamma(t_i)=g_i y$ for $i=0,\ldots,n-1$, it follows $\gamma(t_i)=g_i \lambda y \neq \gamma(t_j)= g_j \lambda y$ whenever $i\neq j$ and so  $g_i \neq g_j$ whenever $i \neq j$. This proves
\[
\mathrm{m}(\lambda y,r,G)\geq n.
\]
Since it holds for any $n \in \mathbb N$, we get
\[
\lim_{\parallel y \parallel +\infty} \mathrm{m}(y,r,G)=+\infty.
\]
If $Ay=y$, then for any $\lambda \in \mathbb R$, we have
\[
d(\lambda y, g^m (\lambda y))^2=m \parallel v \parallel^2,
\]
and so
\[
\lim_{m \mapsto +\infty} d(\lambda y, g^m (\lambda y))^2=+\infty,
\]
for any $\lambda \in \mathbb R$. Hence the above idea works as well and so
\[
\lim_{\parallel y \parallel +\infty} \mathrm{m}(y,r,G)=+\infty.
\]
Summing up, keeping in mind that $m(y,r,K)\leq m(y,r,G)$ whenever $K\subset G$, we have proved the following result.
\begin{proposition}\label{toritto-sostanzioso}
Let $G\subset \mathrm{Iso}(\mathbb R^n,\langle \cdot , \cdot \rangle)$ be a closed subgroup. Let $G^o$ be the connected component containing $e$. Assume that $G^o$ has dimension bigger than one and $(\mathbb R^n)^{G^o}=\{0\}$ or is empty. Then for any $r>0$ we have
\[
\lim_{\parallel y \parallel \mapsto +\infty} \mathrm{m}(y,r,G)=+\infty.
\]
Therefore any open $G$-invariant unbounded subset of $\mathbb R^n$ is compatible with   $G$.
\end{proposition}

\begin{remark}
If $G$ is connected and $(\mathbb R^n)^{G}$ is a closed non trivial subspace, then the $G$ action on $\mathbb R^n$ is not compatible. Indeed, if $v\in  (\mathbb R^n)^{G}$ is non zero, then
 $\mathrm{m}(\lambda v,r,G)=1$ for any $r>0$ and for any nonzero $\lambda \in \mathbb R$ and so $G$ is not compatible.
\end{remark}
\begin{corollary}\label{prodotto}
Let $G_1$ be a closed subgroup of $\mathrm{Iso}(\mathbb R^n,\langle \cdot , \cdot \rangle)$ and let $G_2$ be a closed subgroup of $\mathrm{Iso}(\mathbb R^m,\langle \cdot , \cdot \rangle)$. Assume that both $G_1$ and $G_2$ satisfy the conditions in Proposition \ref{toritto-sostanzioso} or in Proposition \ref{toritto-errenne-discreto-generale}. Then any open $G_1 \times G_2$-invariant unbounded subset of $\mathbb R^{n+m}$  is compatible with respect to $G_1 \times G_2$.
\end{corollary}
We state now the main result in the case of the Euclidean space.

Let $G\subset\mathrm{Iso}(\mathbb R^n,\langle \cdot,\cdot \rangle)$ and let $\Omega \subset \mathbb R^n$ any $G$-invariant open and unbounded subset of $\mathbb R^n$.
Let $H^{1}(\Omega)$ and $H^{1}_{0}(\Omega)$ be the usual Sobolev spaces.

The action of $G$ on $H^1_0(\Omega)$ is defined by
\[
gu(x):=u(g^{-1} x), \quad g\in G
\]
and the subspace of fixed point for this action is
\[
H^1_{0,G}(\Omega):=\left\{u\in H^1_0(\Omega): gu=u, \ \forall g\in G\right\}.
\]
\begin{theorem}\label{toritto-compatto-euclideo}
Let $G\subset \mathrm{Iso}(\mathbb R^n, \langle \cdot, \cdot \rangle)$ be a closed subgroup. Let $\Omega$ be
an unbounded open $G$-invariant subset of $\mathbb R^n$. Assume that one of the following conditions hold:
\begin{enumerate}
\item $G^o$, has dimension bigger than one and $(\mathbb R^n)^{G^o}=\{0\}$ or is empty;
\item $G$ is a discrete non finite closed subgroup of the translation group.
\item $G=G_1 \times G_2\subset \mathrm{Iso}(\mathbb R^{n+m}, \langle \cdot, \cdot \rangle)$ where,
$G_1 \subset \mathrm{Iso}(\mathbb R^{n}, \langle \cdot, \cdot \rangle)$, $G_2 \subset \mathrm{Iso}(\mathbb R^{m}, \langle \cdot, \cdot \rangle)$ and $G_1$ and $G_2$ satisfy condition $1$ or $2$.
\end{enumerate}
Then, the following embedding
\[
H^1_{0,G}(\Omega) \hookrightarrow L^p (\Omega),\ 2 < p < 2^*,
\]
is compact.
\end{theorem}
\begin{proof}
Applying Proposition \ref{toritto-sostanzioso}, Proposition \ref{toritto-errenne-discreto-generale} and Corollary \ref{prodotto}
the proof follows like in  \cite[p.16-17]{willem}.
\end{proof}

\begin{remark}\label{toritto-esempio}
If the condition $(\mathbb R^n)^{G^o}=\{0\}$ or empty is not satisfied then $H^1_{0,G}(\Omega) \hookrightarrow L^p (\Omega),\ 2 < p < 2^*,$ is not compact in general.
Indeed consider in $H^{1}(\mathbb R^{3})$ the action of $G=O(2)\times \mathbb Z_{2}$. The fixed points
for this action are the functions which are radial in the first two variables and odd in the third one,
and $(\mathbb R^n)^{G^o} = \mathbb R$.
However they do not have compact imbedding into $L^{p}(\mathbb R^{3}), 2<p<6$ due to the
invariance by translations in the third variable. A counterexample can be constructed in this way.
Let $\phi,\psi\in C_{0}^{\infty}(\mathbb R)\setminus\{0\}$
with $\psi$ odd and $\textrm{supp } \psi=[-2,-1]\cup[1,2]$.
Define for any $n\geq1$:
\begin{eqnarray*}
u_{n}(x_{1}, x_{2}, x_{3}) =
\begin{cases}
\phi(x_{1}^{2} + x_{2}^{2}) \psi (x_{3} - n) &\mbox{ \ \ if \ \ } x_{3}\geq n, \\
0 &\mbox{ \ \ if \ \  } x_{3}\in[0,n],
\end{cases}
\end{eqnarray*}
and let $\widetilde u_{n}$ be its oddness extension on $x_{3}\leq0$.
Observe that $\widetilde u_{n}\rightharpoonup 0$ in $H^{1}(\mathbb R^{3})$
but
$$\int_{\mathbb R^{3}} |\widetilde u_{n}|^{p}dx =\int_{\mathbb R^{3}}|\phi(x_{1}^{2} + x_{2}^{2})|^{p} |\psi (x_{3} )|^{p}dx=:c>0$$
(with $c$  independent on $n$) and so  $\{\widetilde u_{n}\}$
does not converge to $0$ strongly  into $L^{p}(\mathbb R^{3})$.
\end{remark}
Our aim is to apply the previous results to the existence of nonradial solutions
for some partial differential equations in $\mathbb R^{n}$.
Since these equations are invariant for $\mathrm O(n)$ the main difficulty is exactly
to exclude that the solutions found are radial.
In order to do that  we need to consider the action of a suitable group
in such a way that the subspace of the fixed point for this action
has not radial functions (except of course the zero function).
The next abstract result will be fundamental in order to construct this group.

%

Let $G$ be a Lie group and let $H\subset G$ be a closed subgroup and let $\mathrm{N}(H)=\{g\in G:\, gHg^{-1}=H\}$.
\begin{lemma}\label{toritto-tecnico}
Assume that there exists $\tau\in \mathrm N(H)$ such that $\tau \notin H$ and $\tau^2=e$. Then the subgroup $\widetilde H$ generated by $H$ and $\tau$ is closed, and so it is a Lie group, and $H$ is a normal subgroup of index two. In particular there exists a surjective homomorphism
\[
\rho: \widetilde H \longrightarrow \mathbb Z_2,
\]
such that $\rho(H)=1$ and $\rho(\tau)=-1$.
\end{lemma}
\begin{proof}
It is easy to check that $\widetilde H=\{g\tau^i:\, g\in H, i=0,1\}$ and so  $H$ is a normal subgroup of $\widetilde H$ of index $2$. In particular, the natural projection
\begin{equation}\label{eq:rho}
\rho:\widetilde H \longrightarrow \widetilde H/H \cong \mathbb Z_2,
\end{equation}
defines an homeomorphism satisfying $\rho(H)=1$ and $\rho(\tau)=-1$. Finally, we prove that $\widetilde H$ is closed and so it is a Lie group.

Let $\{h_n\} \subset \widetilde H$ be a sequence converging to some $h_0 \in G$. Up to subsequence we may assume that $\rho(h_n)=1$ or $\rho(h_n )=-1$. This means that $h_n \in H$ or $h_n=s_n \tau$, where $s_n \in H$ and so
$h_0 \in \widetilde H$.
\end{proof}

 To extend some results
on existence of nonradial solutions for partial differential equations in $\mathbb R^{n}$
especially in the cases $n=3$ and $n=5$ (see the Introduction),
 the next  two examples will be useful.
They can be seen as an
application of the previous Lemma \ref{toritto-tecnico} and the next Lemma
\ref{toritto-non-radiali}, to which the reader is explicitly referred.
 However this last lemma will be proved in a more general situation in Section
\ref{sec:SSnoncompact}.
\begin{Ex}\label{ex:fondamentali5}
Let $H=\mathrm{SO}(3) \times \mathrm{SO}(2) \subset \mathrm{O}(5)$ as follows
\[
(A,B) \mapsto \left(\begin{array}{cl} A & 0 \\ 0 & B \end{array}\right).
\]
Let $\tau:\mathbb R^5 \longrightarrow \mathbb R^5$ given by
\[
\tau(x_1,x_2,x_3,x_4,x_5)=(-x_1,-x_2,-x_3,x_4,x_5).
\]
Then $\tau$ is an isometry satisfying $\tau^2=\mathrm{Id}$ and $\tau H \tau =H$. By Lemma \ref{toritto-tecnico} the subgroup $\widetilde H$ generated by $H$ and $\tau$ is closed and there exists a surjective homomorphism $\rho:\widetilde H \longrightarrow \mathbb Z_2$ such that $\rho(H)=1$.  Note that $\widetilde H=\mathrm{O}(3) \times \mathrm{SO}(2)$. By Lemma  \ref{toritto-non-radiali}
(note that  $H$ is compatible by Theorem \ref{toritto-compatto-euclideo}) the action
\[
\widetilde H \times H^1 (\mathbb R^5) \longrightarrow H^1 (\mathbb R^5), \qquad (g,u)\mapsto \rho(g) u(g^{-1} \cdot)
\]
is isometric, the embedding
\[
H^1_{\widetilde H}(\mathbb R^5 ) \hookrightarrow L^p (\mathbb R^5 ),\ 2 < p < 2^*,
\]
is compact and finally $H^1_{\widetilde H}(\mathbb R^5 )$ does not contains radial functions,
except the null one.
\end{Ex}
\begin{Ex}\label{ex:fondamentali3}
Let $H=\mathrm{SO}(2) \times \mathbb Z \subset \mathrm{Iso}(\mathbb R^3, \langle \cdot , \cdot \rangle)$ being
\[
(A,n)  \mapsto \left(\begin{array}{cl} A & 0 \\ 0 & 1 \end{array}\right) + \left[\begin{array}{c} 0 \\ 0 \\ n \end{array}\right].
\]
Pick $\tau:\mathbb R^3 \longrightarrow \mathbb R^3$ given by
\[
\tau(x_1,x_2,x_3)=(-x_1,x_2,x_3).
\]
It is easy to check  $\tau^2=\mathrm{Id}$ and $\tau H \tau=H$. Denote by $\widetilde H$ the subgroup generated by $H$ and $\tau$. Applying Lemma \ref{toritto-tecnico} and Lemma \ref{toritto-non-radiali}
(note again that  $H$ is compatible by Theorem \ref{toritto-compatto-euclideo})
the action
\[
\widetilde H \times H^1 (\mathbb R^3) \longrightarrow H^1 (\mathbb R^3), \qquad (g,u)\mapsto \rho(g) u(g^{-1} \cdot)
\]
is isometric, the embedding
\[
H^1_{\widetilde H}(\mathbb R^3 ) \hookrightarrow L^p (\mathbb R^3 ),\ 2 < p < 2^*,
\]
is compact and finally $H^1_{\widetilde H}(\mathbb R^3 )$ does not contains radial functions,
except of course the null function. Observe that the functions in $H^1_{\widetilde H}(\mathbb R^3 )$ are periodic
in the third variable.
\end{Ex}

\begin{remark}\label{rem:fractional}

Of course  the examples above are easily generalized to the fractional Sobolev space $H^{s}(\mathbb R^{n})$ where $n\in\{3,5\}$ and $2^{*}$
is replaced by the critical exponent $2^{*}_{s} = 2n/(n-2s), n>2s.$
\end{remark}

 \subsection{Existence of nonradial solutions  for some elliptic PDEs.}\label{subsec:appl}
We show now how our method permits to obtain multiplicity results of nonradial solutions
 for  some scalar field equations which enjoy the radial symmetry. Known result are also recovered.
 Just to give an idea of how our method works, as we anticipated into the Introduction,
 we limit ourself to few particular equations.
  Of course many other equations can be treated with the same approach.

\subsubsection{A classical scalar field equation}
In particular we are able to find nonradial solutions for an elliptic equation in the whole space $\mathbb R^{n}, n\geq3$.
Consider the problem
\begin{equation}\label{eq:problema1}
-\Delta u + b(|x|) u =f(|x|,u) \quad \text{ in } \mathbb R^{n}, \ n\geq3
\end{equation}
under the assumptions
\begin{itemize}
\item[(1)]\label{h1} $b\in C([0,+\infty, \mathbb R)$ is bounded from below by a positive constant $a_{0}$.
\item[(2)]\label{h2} $f\in C ([0,+\infty) \times \mathbb R, \mathbb R)$ and there are positive constants $a_{1}, R$ and a constant
$1<q<(n+2)/(n-2)$ such that
$$|f(r,u)| \leq a_{1} |u|^{q} \quad \text{for any } r\geq0, \ |u|\geq R.$$
\item[(3)]\label{h3} There exists $\mu>2$ such that
$$\mu F(r,u):= \mu \int_{0}^{u}f(r,v)dv \leq uf(r,u) \quad \text{for any } r\geq0, \ u\in \mathbb R.$$
\item[(4)] \label{h4} There exists $K>0$ such that $\inf_{r>0, |u|=K} F(r,u)>0$.
\item[(5)] \label{h5} $f(r,u)=o(|u|)$ for $u\to 0$ uniformly in $r\geq0$.
\item[(6)]\label{h6} $f$ is odd in $u: f(r,-u) = - f(r,u)$ for any $r\geq0, u\in \mathbb R$.
\end{itemize}

The above problem has been considered in $\mathbb R^{n}, n=4$ or $n\geq6$, in \cite{BW} which proved the existence of infinitely many solutions in $H^{1}(\mathbb R^{n})$ which are not radial.

We are now able to achieve the same conclusion in $\mathbb R^{n}$
for any $n \geq4$, extending \cite[Theorem 2.1]{BW}.
\begin{theorem}\label{th:BW}
Under the previous assumptions  problem \eqref{eq:problema1} possesses an unbounded sequence of solutions $\{\pm u_{k}\}_{k\in \mathbb N}$ in $\mathbb R^{n}, n\geq4$ which are nonradial.
\end{theorem}

It remains only to show the case $n=5$.
Let
$$\phi(u) = \frac{1}{2}\int_{\mathbb R^{5}} |\nabla u|^{2} +\frac{1}{2}\int_{\mathbb R^{n5}} b(|x|) u^{2} -\int_{\mathbb R^{5}} F(|x|, u)$$
be the energy functional associated to problem \eqref{eq:problema1}
which is well defined and $C^{1}$ on the subspace of $H^{1}(\mathbb R^{5})$
$$X= \left\{ u\in H^{1}(\mathbb R^{5}) : \int_{\mathbb R^{5}} b(|x|)u^{2}<+\infty \right\}.$$
This space is continuously embedded into $H^{1}(\mathbb R^{5})$, then into $L^{p}(\mathbb R^{5})$,
 but not compactly.
However for what we have seen before (see Example \ref{ex:fondamentali5}) 
 the closed and infinite dimensional space
of fixed points of $X$ for the group $\widetilde H$, $X_{\widetilde H}$,
has compact embedding into $L^{p}(\mathbb R^{5}), p\in (2,2^{*})$,
and then, being the functional $\phi$ invariant under the action of $\widetilde H$
(which is important to apply the Palais' Principle)
 the conclusion follows exactly as in \cite{BW}.

\medskip
As a consequence of Example \ref{ex:fondamentali3}  we have also that the problem in $\mathbb R^{3}$ has infinitely many
solutions which are radially symmetric in the first two variables and periodic in the third one;
hence non radial in $\mathbb R^{3}$.

Of course, the fact that the above equation has infinitely many radial solutions
in $\mathbb R^{n}, n\geq3$, is well known and obtained by taking advantage
of the compact embedding of the radial functions.

\bigskip

\subsubsection{A nonlocal fractional scalar field equation}
The above argument also works for the following system of fractional elliptic equations.

Given $\omega>0$, $n\geq 4$,
$\alpha\in (0,n), s\in (0,1)$ and  $1+\alpha/n<p<(n+\alpha)/(n-2s)$ consider the problem
  \begin{equation}\label{eq:problema2}
\begin{cases}
 (-\Delta)^{s} u+\omega u=\varphi |u|^{p-2}u, &
 \quad \text{ in } \mathbb R^{n} \medskip  \\
(-\Delta)^{\alpha/2} \varphi = \gamma(\alpha)|u|^{p},  &\quad \text{ in } \mathbb R^{n}
\end{cases}
\end{equation}
where $\gamma(\alpha):=\frac{\pi^{n/2}2^{\alpha} {\Gamma(\alpha/2)}}{{\Gamma({n}/{2}-{\alpha}/{2})}}$,
$\Gamma$ is the gamma function, and the unknowns $u,\varphi$ are found in the fractional
Sobolev spaces
$$u\in H^{s}(\mathbb R^{n}), \quad \varphi\in \dot{H}^{\alpha/2}(\mathbb R^{n}).$$
The search of solutions for such a problem is reduced to find critical points $u\in H^{s}(\mathbb R^{n})$
of the  following $C^{1}$ functional
$$E_{\omega}(u)=
\int_{\mathbb R^{n}}|(-\Delta)^{s/2}u|^{2}+\frac{\omega}{2}\int_{\mathbb R^{n}}u^{2} -\frac{1}{2p}
\int_{\mathbb R^{n}} \varphi_{u} |u|^{p}, \quad \text{where }  \varphi_{u}=\frac{1}{|\cdot|^{n-\alpha}}*|u|^{p}.
$$
Then in the following we will speak also of ``solution $u$'' of \eqref{eq:problema2}
since $\varphi_{u}$ is univocally determined by $u$ in virtue of the second equation.
See \cite{DSS} for the details, where the problem is addressed and
 where, among other results, the multiplicity of nonradial
solutions has been proved  in the cases $n=4$ and $n\geq6$ by using the symmetric Mountain Pass Theorem.

We are able to obtain a similar result also in $\mathbb R^{5}$.
Recall Remark \ref{rem:fractional}
to deal with  $H^{s}(\mathbb R^{5})$.
Let $\widetilde H$ as in Example \ref{ex:fondamentali5}
and
 $H^{s}_{\widetilde H}(\mathbb R^{5})$
 be
 the closed and infinite dimensional subspace of fixed points for this action.
It is easy to see that if $u\in H^{s}_{\widetilde H}(\mathbb R^{5})$,
then also $\varphi_{u}$ enjoys the same symmetry
and that the functional $E_{\omega}$ is invariant for this action
(and then the Palais' Principle applies).
Since  $H^{s}_{\widetilde H}(\mathbb R^{5})$ has compact embedding into $L^{p}(\mathbb R^{5}), 2<p<2^{*}_{s}$ we can prove the Palais-Smale condition and conclude exactly as in
\cite[Theorem 5.3]{DSS} obtaining explicitly the following result covering the case $n=5$.
\begin{theorem}\label{th:DSS}
Under the previous assumptions problem \eqref{eq:problema2} possesses an unbounded sequence of
solutions $\{\pm u_{k}\}_{k\in \mathbb N}$ in  $\mathbb R^{n}, n\geq4$ which are nonradial.
\end{theorem}

As before, in the case of $\mathbb R^{3}$ we can achieve existence of solutions
radially symmetric in the first two variables and  periodic in the third one.

Again the existence of infinitely many radial solutions in $\mathbb R^{n}, n\geq3$, is known
and obtained in a standard way by using the compact embedding of the radial functions.

\bigskip

\subsubsection{A Schr\"odinger-Poisson system}
With our approach we are able to extend a result of d'Avenia \cite{P}
in $\mathbb R^{n}, n\geq4$ with a simpler proof by using the compact embedding
of our working space.
Let us start with the case $n=5$.
Following \cite{P} we consider
  \begin{equation*}
\begin{cases}
 -\frac{1}{2}\Delta u+ \omega u+\varphi u = |u|^{p-2}u, &
 \quad \text{ in } \mathbb R^{5}, \medskip  \\
-\Delta \varphi = \gamma(2) u^{2},  &\quad \text{ in } \mathbb R^{5}
\end{cases}
\end{equation*}
in the unknowns $u,\varphi:\mathbb R^{5}\to \mathbb R$ and $\omega>0$.
The constant $\gamma(2)$ is defined as above.
Here $p\in(4,6)$ is given.
Then after rescaling, the above problem is equivalent to find
solutions
$$u\in H^{1}(\mathbb R^{5}), \ \ \varphi\in D^{1,2}(\mathbb R^{5})\ \  \text{\ not radial, and } \ \lambda>0$$
for the system
  \begin{equation}\label{eq:problema3}
\begin{cases}
 -\frac{1}{2}\Delta u+ u+\varphi u = \lambda|u|^{p-2}u, &
 \quad \text{ in } \mathbb R^{5} \medskip  \\
-\Delta \varphi = \gamma(2) u^{2},  &\quad \text{ in } \mathbb R^{5}.
\end{cases}
\end{equation}
Equivalently the problem is reduced to find a critical point, which is not radially symmetric, of
the functional
$$J(u) = \frac{1}{4}\int_{\mathbb R^{5}} |\nabla u|^{2} +\frac{1}{2}\int_{\mathbb R^{5}} u^{2} + \int_{\mathbb R^{5}}\varphi_{u} u^{2}, \ \text{ where } \ \varphi_{u} = \frac{1}{|\cdot|^{3}} * u^{2},
$$
restricted to the sphere $$S=\left\{ u\in H^{1}(\mathbb R^{5}): \int_{\mathbb R^{5}} |u|^{p} =1\right\}$$
and $\lambda$ is the associated Lagrange multiplier. For this reason (as before, since $\varphi$ is uniquely determined by $u$) we will speak
of ``solutions $u$ and $\lambda$'' of \eqref{eq:problema3}.

With our approach we can restrict the functional to
the subspace of fixed points for the action considered in Example \ref{ex:fondamentali5}, i.e.
$H^{1}_{\widetilde H}(\mathbb R^{5})$, that we know it has compact embedding into $L^{p}(\mathbb R^{5})$.
Then the set
$$S_{\widetilde H}=\left\{ u\in H_{\widetilde H}^{1}(\mathbb R^{5}): \int_{\mathbb R^{5}} |u|^{p} =1\right\}$$
is weakly closed.
 It is easy to see  that if $u\in H^{1}_{\widetilde H}$ then also $\varphi_{u}: =\frac{1}{|\cdot|^{3}} * u^{2}$
 enjoys the same symmetry
 (again, important to apply the Palais's Principle).
Moreover  the energy functional restricted to $S_{\widetilde H}$
 is   bounded below
 and coercive (as proved in \cite{P}); but also, weakly lower semicontinuous
 and satisfies the Palais-Smale condition, as it is standard to see by using the compact embedding
 just stated.

 Then not only the existence of a minimizer is guaranteed, but also,
since the functional is even, by  the Ljusternick-Schnirelmann theory
we deduce the following

\begin{theorem}\label{th:P}
Problem \eqref{eq:problema3} has infinitely many solutions $(u_{n}, \lambda_{n})$
with $u_{n}$ nonradial.
\end{theorem}

The above theorem  for the  cases $n=4$ and $n\geq6$ can be easily obtained by the approach of \cite{BW}.

Of course our methods permits to have solutions in $\mathbb R^{3}$
which are radial in the first two variables and periodic in the third one.

Again, also in this case the existence of radially symmetric solutions is well-known, for any $n\geq3$.
\section{The case of Hadamard manifolds}\label{sec:Hadamard}
Let $(M,g)$ be an Hadamard manifold of $\dim M=n$. This means $(M,g)$ is a simply connected complete Riemannian manifold of non positive sectional curvature, see  \cite{eberlein, helgason}. Hence if $p\in M$, then $\exp_p: T_p M \longrightarrow M$ is a diffeomorphism. Since $M$ is simply connected it is orientable. Let $\{U_i,\phi_i\}_{i\in I}$ be an orientable atlas. This means that for every $i\in I$,  $U_i \subset \mathbb R^n$ and $\phi_i : U_i \longrightarrow \widetilde U_i\subset M$ is a positive diffeomorphism. The Riemannian volume form $\nu$ is the smooth $n$-form such that
\[
\phi_i^* \nu=\sqrt{\det G} \mathrm d x_1 \wedge \cdots \wedge \mathrm d x_n,
\]
where $G=\left( g( \frac{\partial}{\partial x_i}, \frac{\partial}{\partial x_j}) \right)_{1\leq i,j\leq n}$ is the matrix associated to $g$ with respect to the basis $\{\frac{\partial}{\partial x_1},\ldots,\frac{\partial}{\partial x_n} \}$. It is easy to check that any isometry $T:M\longrightarrow M$ satisfies $T^* \nu=\nu$, i.e., $T$ preserves the Riemannian volume form $\nu$.

Let $G\subset\mathrm{Iso}(M,g)$ and let $\Omega \subset M$ any $G$-invariant open and unbounded subset of $M$.
Let $H^{1}(\Omega)$ and $H^{1}_{0}(\Omega)$ be the usual Sobolev spaces defined
as for the case of $\mathbb R^{n}$, see \cite{aubin}.

The action of $G$ on $H^1_0(\Omega)$ is defined by
\[
gu(x):=u(g^{-1} x), \quad g\in G
\]
and the subspace of fixed point for this action is
\[
H^1_{0,G}(\Omega):=\left\{u\in H^1_0(\Omega): gu=u, \ \forall g\in G\right\}.
\]

Applying Proposition \ref{toritto-discreto-generale} we have the following result
which give the compactness of the Sobolev embedding in the case of a Riemannian manifold.
\begin{theorem}\label{toritto-discreto-compattezza}
Let $(M,g)$ be an Hadamard manifold and
let $G\subset\mathrm{Iso}(M,g)$ be a closed and discrete subgroup with infinite elements. If $M/G$ is compact, then the following embedding
\[
H^1_{0,G}(M) \hookrightarrow L^p (M),\ 2 < p < 2^*,
\]
is compact.
\end{theorem}
\begin{proof}
Indeed having Proposition \ref{toritto-discreto-generale} the proof  follows as in \cite[p.16-17]{willem}.
\end{proof}
Let $G\subset \mathrm{Iso}(M,g)$ be a compact subgroup. By a Theorem of Cartan, see \cite{helgason}, $G$ has a fixed point. Let $p\in M^G=\{q\in M:\, G(q)=q\}$ be a fixed point. The isotropy representation
\[
G \longrightarrow \mathrm{O}(T_p M), \qquad k \mapsto \mathrm d k_p.
\]
is injective and it satisfies
\[
k(\exp_p(v))=\exp(\mathrm d k_p (v)).
\]
Therefore the exponential map at $p$ is $G$-equivariant, i.e., it interchanges the $G$ action on $M$ with the $G$ action on $T_p M$. In the sequel we also denote by $kv=\mathrm d k_p (v)$.  It is well-known that $M^G$ is a totally geodesic submanifold of $M$ and $T_p M^G =(T_p M)^{G}$, see \cite{renatone-marcox}. Hence if the $G$ action satisfies the first condition of Theorem \ref{toritto-compatto-euclideo}, keeping in mind that any point can be joined by a unique minimizing geodesic, then $M^{G}=\{p\}$. The vice-versa holds as well.
\begin{proposition}\label{toritto-varieta}
Let $(M,g)$ be an Hadamard manifold and let $G\subset \mathrm{Iso}(M,g)$ be a compact connected subgroup of dimension bigger than one.  Assume that $M^{G}=\{p\}$. Then for any $r>0$, we have
\[
\lim_{d (p,z) \mapsto +\infty} \mathrm{m}(z,r,G)=+\infty.
\]
Therefore any $G$-invariant open unbounded subset of $M$ is compatible.
\end{proposition}
\begin{proof}
Since $(M,g)$ has negative curvature, applying the Rauch Theorem \cite{docarmo}, if $v,w\in T_p M$, then
\[
\parallel x-y\parallel \leq  d (\exp_p (v),\exp_p (w)).
\]
Let $r>0$ and let $n\in \mathbb N$. By Proposition \ref{toritto-sostanzioso},
\[
\lim_{\parallel y \parallel \mapsto +\infty} \mathrm{m}(y,r, G)=+\infty.
\]
Let  $y\in T_p M$ and let $n \in \mathbb N$ be such that there exist $g_1,\ldots,g_n \in G$ such that $\parallel g_i y- g_j y \parallel >2r$ for $i\neq j$ and so $B_{r}(g_i y) \cap B_{r}(g_i y) =\emptyset$ for $i\neq j$. Since
\[
d (\exp_p (g_i y), \exp_p (g_j y))=d(g_i \exp_p (y),g_j \exp_p (y)) \geq \parallel g_i y - g_j y \parallel >2r,
\]
it follows $\mathcal B_{r}(g_i \exp_p (y)) \cap \mathcal B_{r}(g_i \exp_p (y)) =\emptyset$ for $i\neq j$ and so, keeping in mind $d (p,\exp_p (y))=\parallel y \parallel$, the result follows.
\end{proof}
Aa an application, we have the following compactness result for an Hadamard manifold.
\begin{theorem}\label{toritto-compatto-hadamard}
Let $(M,g)$ be an Hadamard manifold and let $G\subset \mathrm{Iso}(M,g)$ be a compact subgroup of dimension bigger or equal than one. Assume that $M^G=\{p\}$. Let $\Omega$ be
an unbounded  $G$-invariant open subset of $M$ containing $p$. Then, the following embedding
\[
H^1_{0,G}(\Omega) \hookrightarrow L^p (\Omega),\ 2 < p < 2^*,
\]
is compact.
\end{theorem}
\begin{proof}
It is similar to \cite{willem} p. $16-17$.
\end{proof}
Let $p\in M$ and let $f:M \longrightarrow \mathbb R$ be a function. We say that $f$ is a radial function
(with respect to $p$) if $f(z)=f(w)$ whenever $d(p,z)=d(p,w)$. Since $\exp_p : T_p M \longrightarrow  M$ is a diffeomorphism, and keeping in mind that $d(p,z)=\parallel \exp_p^{-1} (z) \parallel$, see \cite{docarmo}, it follows that $f:M\longrightarrow \mathbb R$ is a radial function if and only if $f\circ \exp_p^{-1}:T_p M \longrightarrow \mathbb R$ is a radial function.

Let $G\subset G_p=\{g\in \mathrm{Iso}(M,g):\, gp=p\}$ be a closed subgroup. Then $H^1_{0,G}(M)$ contains the set of radial functions. Indeed, if $f:M \longrightarrow \mathbb R$ is a radial function then it is $G$-invariant due to the fact that for any $k\in G$ we have
\[
d(p,kz)=d(kp,kz)=d(p,z).
\]
\section{Symmetry and compactness for symmetric spaces of noncompact type}\label{sec:SSnoncompact}
Let $(M,g)$ be an Hadamard manifold. Let $p\in M$ and let $H \subset G_p$ be closed subgroup. Assume that there exists a closed subgroup $\widetilde H \subset G_p$ such that $H\subset \widetilde H \subset G_p$ and $\widetilde H /H$ has two elements. Then $H$ is a normal subgroup of $\widetilde H$ and the natural projection $\rho:\widetilde H \longrightarrow \widetilde H /H \cong \mathbb Z_2$ is a surjective homomorphism. By Lemma \ref{toritto-tecnico} this holds if there exists $\tau \in G_p$ such that $\tau \notin H$, $\tau^2=e$ and $\tau H \tau =H$.  Indeed the subgroup $\widetilde H$ generated by $H$ and $\tau$ is closed and $\widetilde H/ H$ has two elements.

Let $\Omega$ be
an unbounded open $\widetilde H$-invariant subset of $M$. Define
\[
\widetilde H \times H^1 (\Omega) \longrightarrow H^1(\Omega), \qquad (g,u) \mapsto  \rho(g) u(g^{-1} \cdot)=g u.
\]
It is easy to check that this map defines an isometric action of $\widetilde H$ on $H^1(\Omega)$.
Denote by $H^1_{0,\widetilde H} (\Omega)=\{u\in H^1(\Omega):\, gu=u$ for any $g\in \widetilde H\}$. Since $\rho(h)=1$ if $h\in H$, then
$ H^1_{0,\widetilde H} (\Omega)$ is a closed subspace of $H^1_{0,H}(\Omega)$.
\begin{lemma}\label{toritto-non-radiali}
Under the above assumption, if
\[
H^1_{0,H}(\Omega) \hookrightarrow L^p (\Omega),\ 2 < p < 2^*,
\]
is compact, then the embedding
\[
H^1_{0,\widetilde H}(\Omega) \hookrightarrow L^p (\Omega),\ 2 < p < 2^*,
\]
is compact and the set $H^1_{0,\widetilde H}(\Omega)$ does not contain radial functions unless $u=0$.
\end{lemma}
\begin{proof}
Since  $H^1_{0,\widetilde H}(\Omega)$ is a closed subspace of $ H^1_{0,H}(\Omega)$ it follows that  the embedding
\[
H^1_{0,\widetilde H}(\Omega) \hookrightarrow L^p (\Omega),\ 2 < p < 2^*,
\]
is the restriction of a compact operator on a closed subspace and so it
is compact as well. Finally, pick $\tau \notin H$. If $u\in H^1_{0,\widetilde H} (\Omega)$, then
\[
\tau u(\exp_p (v))=-u(\exp_p (\tau (v))
\]
and so $u$ is not radial unless it is zero.
\end{proof}
The aim of this section is to generalize the idea due to Bartch and Willem \cite{BW} in order to
show the existence of group actions whose subspace
of fixed points have nonradial functions.
\subsection{Symmetric Spaces of noncompact type of rank one}
Let $(M,g)$ be a symmetric space of noncompact type of rank one. Let $p\in M$ and let $G_p =\{g\in \mathrm{Iso}(M,g): g(p)=p\}$. It is a well known $G_p$ is a compact group acting  transitively on the unit sphere of
$(T_p M, g(p))$, see \cite{helgason}. Hence $G_p$ satisfies the assumption on Theorem \ref{toritto-compatto-hadamard}. We claim that   $H^1_{0,G_p}(M)$ coincides with the set of radial functions of class $H^1_{0} (M)$ that we denote by $H^1_{0,\mathrm{rad},p}(M)$.
Indeed, let  $u:M \longrightarrow \mathbb R$ be a $G_p$-invariant functions. It is enough to prove that $\widetilde u=u\circ \exp_p^{-1}$ is a radial function.

Let $v\in T_p M$ and let $w\in T_p M$ be such that $\parallel v \parallel=\parallel w  \parallel$. Then there exists $g\in G_p$ such that $g^{-1}(v)=w$. Hence
\[
\widetilde u(w)=u(\exp_p ( w))=u\exp_{g^{-1} p}(g^{-1} v) )=g u (\exp_p (v))=\widetilde u (v).
\]
Applying Theorem \ref{toritto-compatto-hadamard} we have proved the following result.
\begin{theorem}
Let $(M,g)$ be a symmetric space of noncompact type of rank one and let $p\in M$. Then, the following embedding
\[
H^1_{0,\mathrm{rad},p} (M) \hookrightarrow L^p (\Omega),\ 2 < p < 2^*,
\]
is compact.
\end{theorem}
Let $\mathbb H^n$ be  the hyperbolic space. $\mathbb H^n$ is a symmetric space of noncompact type of rank one \cite{helgason}. One model is given as follows, see \cite{docarmo}.

Let $\mathbb R^{n+1}=\mathbb R^n \oplus \mathbb R$ endowed by the Minkoswky inner product defined by the quadratic form
\[
q(x,t)=\parallel x \parallel^2 - t^2,
\]
where $\parallel \cdot \parallel $ is the norm with respect to the canonical scalar product on $\mathbb R^n$. Let
\[
\mathbb H^n=\{(x,t)\in \mathbb R^{n+1}: q(x,t)=-1\}.
\]
It is well known that the Minkowsky inner product induces on $\mathbb H^n$ a complete Riemannian metric $\langle \cdot, \cdot \rangle$ with constant sectional curvature $-1$. Moreover, $\mathrm{Iso}(\mathbb H^n, \langle \cdot , \cdot \rangle)=\mathrm{O}(1,n)=\{A\in \mathrm{Gl}(n+1,\mathbb R):\, q(A(x,t))=q(x,t)\}$ and $\mathrm{O}(1,n)_{e_{n+1}}=\mathrm{O}(n)$, where $e_{n+1}=(0,\ldots,0,1)^T$. Moreover the slice representation, i.e., the $\mathrm{O}(n)$ action on $T_{e_{n+1}} \mathbb H^n =\mathbb R^n$ is the standard $\mathrm{O}(n)$ action on $\mathbb R^n$.

Assume that $n\geq 4$. Let $H=\mathrm{SO}(2) \times \mathrm{SO}(n-2) \subset \tilde H=\mathrm{O}(2) \times \mathrm{SO}(n-2)\subset \mathrm{O}(n)$. Then both $H$ and $\tilde H$ satisfy the condition of Theorem \ref{toritto-compatto-hadamard} and $\tilde H /H \cong \mathbb Z_2$. Therefore, the action
\[
\widetilde H \times H^1 (\mathbb H^n) \longrightarrow H^1(\mathbb H^n ), \qquad ((A,B),u) \mapsto  \mathrm{det}(A) u((A^{-1} \cdot,B^{-1} \cdot)) .
\]
is an isometric action, the embedding
\[
H^1_{0,\widetilde H}(\mathbb H^n) \hookrightarrow L^p (\mathbb H^n),\ 2 < p < 2^*,
\]
is compact and the set $H^1_{0,\widetilde H}(\mathbb H^n)$ does not contain radial functions unless $u=0$.

Let $M=\mathrm{SU}(1,n)/\mathrm{S}(\mathrm{U}(1) \times \mathrm{U}(n))$ be the symmetric space of noncompact of type which is the dual of the complex projective space. It is a symmetric space of noncompact type of rank one \cite{helgason}. The isotropy representation is given by the natural action of $\mathrm{U}(n)$ on $\mathbb C^n$.

Let $T\subset \mathrm{SU}(n)$  be a maximal torus.
It is easy to check that $(\mathbb C^n)^{T}=\{0\}$ and so it satisfies the condition of Theorem \ref{toritto-compatto-hadamard}.
It is also well-known that the Weyl group of $\mathrm{SU}(n)$ is isomorphic to the group of permutations of $n-1$ elements, see \cite{mimura-toda}.
Therefore, there exists $\tau \in \mathrm{SU}(n) \setminus\{T\}$ such that $\tau^2=\mathrm{Id}$ and $\tau T\tau=T$. Denote by $\widetilde H$ be the closed group generated by $T$ and $\tau$ and by $\rho: \widetilde H \longrightarrow \widetilde H /T \cong \mathbb Z_2$ the  natural projection. Then the action
\[
\widetilde H \times H^1 (M) \longrightarrow H^1(M), \qquad (g) \mapsto  \rho(g) u(g^{-1} \cdot) .
\]
is an isometric action, the embedding
\[
H^1_{0,\widetilde H}(M) \hookrightarrow L^p (M),\ 2 < p < 2^*,
\]
is compact and the set $H^1_{0,\widetilde H}(M)$ does not contain radial functions unless $u=0$.

Let $M=\mathrm{Sp}(1,n)/\mathrm{Sp}(1)\times \mathrm{Sp}(n))$ be the symmetric space of noncompact type which is the dual of the Quaternionic projective space. It is a symmetric space of noncompact type of rank one, see \cite{helgason}. The isotropy representation is given by the natural action of $\mathrm{Sp}(1) \cdot \mathrm{Sp}(n)$ on $\mathbb H^n$.

Let $T\subset \mathrm{Sp}(n)$  be a maximal torus.
It is easy to check that $(\mathbb H^n)^{T}=\{0\}$ and so it satisfies the condition of Theorem \ref{toritto-compatto-hadamard}.
It is also well-known that the Weyl group of $\mathrm{Sp}(n)$ is isomorphic to $\mathrm{S}_n \ltimes \mathbb Z_2$, see \cite{mimura-toda}.
Therefore, there exists $\tau \in \mathrm{Sp}(n) \setminus\{T\}$ such that $\tau^2=\mathrm{Id}$ and $\tau T\tau=T$. Denote by $\widetilde H$ be the closed group generated by $T$ and $\tau$ and by $\rho: \widetilde H \longrightarrow \widetilde H /T \cong \mathbb Z_2$ the  natural projection. Then the action
\[
\widetilde H \times H^1 (M) \longrightarrow H^1(M), \qquad (g) \mapsto  \rho(g) u(g^{-1} \cdot) .
\]
is an isometric action, the embedding
\[
H^1_{0,\widetilde H}(M) \hookrightarrow L^p (M),\ 2 < p < 2^*,
\]
is compact and the set $H^1_{0,\widetilde H}(M)$ does not contain radial functions unless $u=0$.
\subsection{The space $M=\mathrm{SL}(n,\mathbb R)/\mathrm{SO}(n)$}
The manifold $M=\mathrm{SL}(n,\mathbb R)/\mathrm{SO}(n)$ is a symmetric space of noncompact type. For a sake of completeness we briefly recall some well known facts.
The Cartan decomposition of the Lie algebra of $\mathrm{SL}(n,\mathbb R)$ is given by
\[
\mathfrak{sl}(n,\mathbb R)=\mathfrak{so} (n) \oplus \mathrm{Sym}_0 (n),
\]
where $\mathrm{Sym}_0 (n)=\{A\in \mathfrak{gl}(n,\mathbb R): A=A^T,\, \mathrm{Tr}(A)=0\}$ and $\mathfrak{so}(n)$ is the Lie algebra of $\mathrm{SO}(n)$. Therefore we may identify $T_{[\mathrm{SO}(n)]} M$ with $\mathrm{Sym}_0(n)$ and the isotropy representation of $\mathrm{SO}(n)$ is given by
\[
\sigma:\mathrm{SO}(n) \longrightarrow \mathrm{SO}(T_p M), \qquad A \mapsto \mathrm{Ad}(A),
\]
where $\mathrm{Ad}(A)(X)=AXA^T$, i.e., the adjoint action. This action is isometric with respect to scalar product
$\langle X,Y \rangle:=\mathrm{Tr}(XY)=\mathrm{Tr}(XY^T)$ defined on $\mathrm{Sym}_0$. This allow us to define a Riemannian metric on $M$,
which coincides with $\langle \cdot,\cdot \rangle$ at $[\mathrm{SO}(n)]$ by requesting that the left translation
$L_g :M \longrightarrow M$, $h\mathrm{SO}(n)] \mapsto gh\mathrm{SO}(n)]$ is an isometry.
We also denote by $\langle \cdot,\cdot \rangle$ this Riemannian metric and   $(M,\langle \cdot,\cdot \rangle)$ is a symmetric space of noncompact type of nonpositive sectional curvature \cite{helgason}.
Let $\mathrm{SU}(n) \subset \mathrm{SO}(2n)$ being
\[
A+iB \mapsto \left[\begin{array}{cc} A & B \\ -B & A \end{array}\right].
\]
This means that
\[
\mathrm{SU}(n)=\{C\in \mathrm{SO}(2n):\, CJC^T=J\},
\]
where
$
J=\left[\begin{array}{cc} 0 & -\mathrm{Id}_n \\ \mathrm{Id}_n & 0 \end{array}\right].
$
We claim that $\mathrm{Sym}_0 (2n)^{\mathrm{SU}(n)}=\{0\}$. Indeed, let $A\in \mathrm{Sym}_0 (2n)^{\mathrm{SU}(n)}$. Since $A$ is symmetric it can be diagonalize. Since $A$ is a $\mathrm{SU}(n)$ fixed point, it follows that any eigenspace is preserved by $\mathrm{SU}(n)$. Since $\mathrm{SU}(n)$ acts irreducibly on $\mathbb R^{2n}$ it follows that $A$ must be diagonal with trace $0$ and so  $A=\{0\}$.

Let
$\tau=\left[\begin{array}{cc} \mathrm{Id}_n & 0 \\ 0 & -\mathrm{Id}_n \end{array}\right]$. It is easy to see that $\tau^2=\mathrm{Id}$, $\tau \notin \mathrm{SU}(n)$ and $\tau\mathrm{SU}(2n) \tau=\mathrm{SU}(2n)$.
Identify $\tau$ with $\sigma(\tau)$, we define  $\widetilde \tau=\exp \circ \tau \exp^{-1}$.
Therefore $\widetilde \tau$ is an isometry of $M$, due to the fact that  $\tau$ lies in the image of the isotropy representation, satisfying  $\widetilde \tau^2=\mathrm{Id}_M$. Moreover, if we consider $H$ acting on
$M$, we have  and $\widetilde \tau H \widetilde \tau=H$. Indeed, denoting by $p=[\mathrm{SO}(2n)]$,  for any $h\in H$, we have $\widetilde \tau h \widetilde \tau p=p$ and so $\widetilde \tau h \widetilde \tau $ is completely determine by its differential. Since $\tau \mathrm d h \tau \in H$ it follows that $\widetilde \tau h \widetilde \tau \in H$ as well.

By Lemma \ref{toritto-tecnico} the subgroup $\widetilde H$ generated by $H$ and $\tau$ is closed and there exists a surjective homomorphism $\rho:\widetilde H \longrightarrow \mathbb Z_2$ such that $\rho(H)=1$.  By Lemma  \ref{toritto-non-radiali}  the action
\[
\widetilde H \times H^1 (\mathrm{SL}(2n,\mathbb R) / \mathrm{SO}(2n)) \longrightarrow H^1(\mathrm{SL}(n,\mathbb R) / \mathrm{SO}(2n)), \qquad (g,u)\mapsto \rho(g) u(g^{-1} \cdot)
\]
is isometric,
the embedding
\[
H^1_{0,\widetilde H}(\mathrm{SL}(2n,\mathbb R) / \mathrm{SO}(2n)) \hookrightarrow L^p (\mathrm{SL}(2n,\mathbb R) / \mathrm{SO}(n) ),\ 2 < p < 2^*,
\]
is compact and finally $H^1_{0,\widetilde H}(\mathrm{SL}(2n,\mathbb R) / \mathrm{SO}(2n) )$ does not contains radial functions unless $u=0$.

\end{document}